\documentclass[8pt]{article}
\usepackage{lineno,hyperref}
\usepackage{amssymb, amsthm, amsmath}
\usepackage{mathtools}
\usepackage{graphics}
\usepackage{caption}
\usepackage{subcaption}

\theoremstyle{plain}
\newtheorem{corollary}{Corollary}[section]
\newtheorem{theorem}[corollary]{Theorem}
\newtheorem{lemma}[corollary]{Lemma}

\theoremstyle{definition}
\newtheorem{definition}[corollary]{Definition}
\newtheorem{conjecture}[corollary]{Conjecture}

\title{Towards a characterization of convergent sequences of $P_n$-line graphs}
\author{Alvaro Carbonero$^1$ \\ $^1$University of Nevada, Las Vegas \\ $^1$Corresponding Author: carboa1@unlv.nevada.edu}

\date{\today}

\begin{document}

\maketitle

\begin{abstract}
Let $H$ and $G$ be graphs such that $H$ has at least 3 vertices and is connected. The $H$-line graph of $G$, denoted by $HL(G)$, is that graph whose vertices are the edges of $G$ and where two vertices of $HL(G)$ are adjacent if they are adjacent in $G$ and lie in a common copy of $H$. For each nonnegative integer $k$, let $HL^{k}(G)$ denote the $k$-th iteration of the $H$-line graph of $G$. We say that the sequence $\{ HL^k(G) \}$ converges if there exists a positive integer $N$ such that $HL^k(G) \cong HL^{k+1}(G)$, and for $n \geq 3$ we set $\Lambda_n$ as the set of all graphs $G$ whose sequence $\{HL^k(G) \}$ converges when $H\cong P_n$. The sets $\Lambda_3, \Lambda_4$ and $\Lambda_5$ have been characterized. To progress towards the characterization of $\Lambda_n$ in general, this paper defines and studies the following property: a graph $G$ is minimally $n$-convergent if $G\in \Lambda_n$ but no proper subgraph of $G$ is in $\Lambda_n$. In addition, prove conditions that imply divergence, and use these results to develop some of the properties of minimally $n$-convergent graphs.
\end{abstract}



\section{Introduction}

In this paper all graphs are finite, simple, and undirected. 

Let $H$ and $G$ be graphs such that $H$ is a connected graph of order at least 3, and $G$ is a nonempty graph. Two edges $e$ and $f$ in a graph $G$ are said to be $H$-adjacent if the edges are adjacent and lie in a common subgraph isomorphic to $H$. Define the $H$-line graph of $G$, or $HL(G)$, as that graph whose vertices are the edges of $G$ and where two vertices of $HL(G)$ are adjacent if they are $H$-adjacent in $G$. Figure \ref{fig1} shows an example of graphs $G$ and $HL(G)$, where $H\cong P_5$. Notice that the edges $e_1$ and $e_2$ are adjacent and lie in a $P_5$ in $G$. By definition, it follows that $e_1$ and $e_2$, as vertices, are adjacent in $HL(G)$. On the other hand, edges $e_2$ and $e_3$ are adjacent in $G$ but do not lie in any common $P_5$. This leads to $e_2$ and $e_3$, as vertices, not being adjacent in $HL(G)$. 

\begin{figure}[!h]
\begin{center}
\includegraphics[width=40mm]{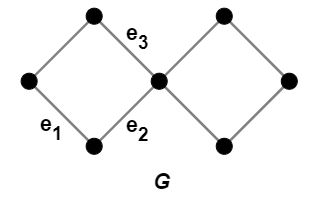}
\includegraphics[width=58mm]{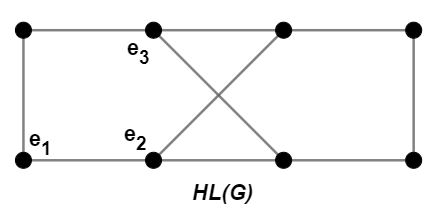}
\end{center}
\caption{A graph and its $H$-line graph when $H\cong P_5$.}
\label{fig1}
\end{figure}

For $k \geq 0$, define $HL^{k+1}(G) = HL(HL^k(G))$ where $HL^0(G) = G$. The sequence $\{HL^k(G) \}$ is said to converge if there exists an integer $N$ such that $HL^N(G)\cong HL^{N+1}(G)$. If the empty graph is part of the sequence, then the sequence is said to terminate. If the sequence does not converge nor terminate, then the sequence is said to diverge. Further, we call a graph $F$ a limit graph if $F\cong HL(F)$. Figure \ref{fig1.2} shows a graph $G$ that is a limit graph when $H\cong P_{10}$. For convenience, if the sequence $\{HL^k(G)\}$ converges and it is understood from context what $H$ we are referring to, we simply say that $G$ has a convergent sequence. 

\begin{figure}[!h]
\begin{center}
\includegraphics[width=60mm]{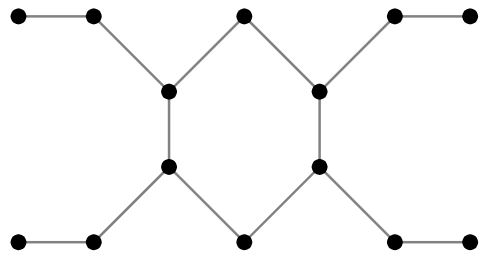}
\end{center}
\caption{A limit graph with $H\cong P_{10}$.}
\label{fig1.2}
\end{figure}

Some special cases of $H$ have been studied. The one studied the most by far is when $H\cong P_3$. If such is the case, then $HL(G) = L(G)$, the well-known line graph of $G$. With the exception of $P_3$, most of the research surrounding $H$-line graphs pertain the characterization of graphs with convergent sequences. In \cite{main}, Chartrand \textit{et. al.} proved that no graph $G$ has a convergent sequence when $H\cong K_{1,n}$ for $n\geq 3$ or when $H\cong K_n$ for $n \geq 4$. In \cite{hisc3}, Jarrett proved that $G$ has a convergent sequence when $H\cong C_3$ if and only if $C_3$ is a subgraph of $G$. In \cite{hisc4}, Chartrand \textit{et. al.} proved that if $G$ is a graph such that $C_4\subseteq G$ but $G$ contains no subgraph isomorphic to $K_1+P_4$, $P_3\times K_2$, $K_{2,3}$ or $K_4$, then $G$ has a sequence that converges when $H\cong C_4$. However, a counterexample of the converse is also provided. Note that each of these four graphs have $C_4$ as a subgraph. This demonstrates that the result in \cite{hisc3} for $C_3$ does not generalize easily.

This paper will focus on the case when $H\cong P_n$ for $n \geq 4$. We ignore the case where $n = 3$ because this is the case of the line graph, making this case vastly different from others. \textbf{From now on assume that $H\cong P_n$}. Define $\Lambda_n$ as the set of all graphs $G$ whose sequence converges. In \cite{chartrandthesis}, Chartrand proved that $\Lambda_3$ is composed of graphs whose components are cycles or $K_{1,3}$. In \cite{main}, Chartrand \textit{et al.} proved that $\Lambda_4$ and $\Lambda_5$ are composed of graphs whose components are cycles of order at least $4$ or $5$, respectively, and the graphs in Figure \ref{fig1.4}. Characterizing $\Lambda_n$ in general becomes harder as $n$ increases because new types of behaviour become possible. For example, Britto-Pacumio in \cite{limits} found and studied disconnected graphs with convergent sequences that had components which did not have convergent sequences. See Figure \ref{fig1.3} for an example found in \cite{limits}. This complex behaviour does not happen when $n = 4, 5$, and so the proofs that characterize $\Lambda_4$ and $\Lambda_5$ are difficult to replicate for a general $n$. 

\begin{figure}
    \centering
    \includegraphics[scale=0.6]{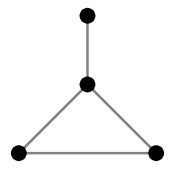}
    \ \ \ \ \ \ \ \ \ \ \ \ \ \ \ \ \ \ \ \ \ 
    \includegraphics[scale=0.5]{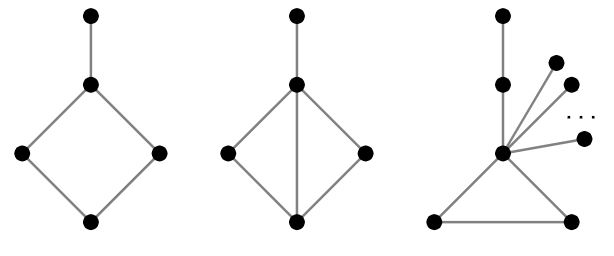}
    \caption{To the left, one graph in  $\Lambda_4$, and to the right, three graphs in $\Lambda_5$.}
    \label{fig1.4}
\end{figure}

\begin{figure}[!h]
\begin{center}
\includegraphics[scale=0.5]{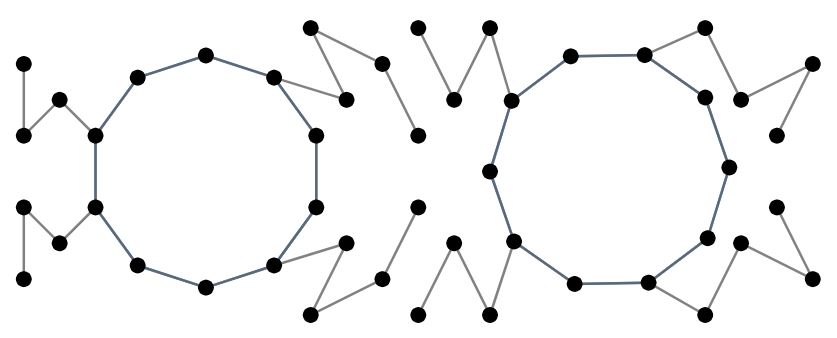}
\end{center}
\caption{A disconnected limit graph when $H\cong P_{16}$, found in \cite{limits}.}
\label{fig1.3}
\end{figure}

To develop a route towards the characterization of $\Lambda_n$, we will study a subset of this set. We say that $G$ is minimally $n$-convergent if $G\in \Lambda_n$ but every proper subgraph of $G$ is not in $\Lambda_n$. Further, let $\lambda_n$ be the set of all graphs with this property. The content of this paper is separated into two parts. The first one deals with conditions that imply divergence, along with a way to study this behaviour. The second part deals with the properties of minimally $n$-converges with results proven by using the ones developed in the first part. At the end of the second part, we also provide a small summary of the ways in which the study of minimally $n$-convergent graphs can progress.

\section{Conditions that imply divergence}

Knowing the conditions that make a sequence diverge facilitates the study of sequences that do converge as they provide the properties that need to be avoided.

We start by categorizing divergence. Although not obvious at first, there are two kinds of divergence. The first is divergence by order. The sequence of a graph $G$ diverges by order if for every positive integer $N$, there exists an integer $k$ such that $|V(HL^k(G))| \geq N$. The second kind of divergence is when the order is bounded yet the sequence of $G$ does not converge. It is easy to generate graphs with the first kind of divergence. The second kind is more difficult to obtain. In fact, the first paper on iterated $H$-line graph sequences, \cite{main}, conjectured that the second kind does not exist. Not much is known about the second kind of divergence, but we do know that it exists. As mentioned above, the connected graphs $G$ such that $G\not \cong HL(G)$ but $G\cong HL^2(G)$ presented in \cite{limits} are an example of a graph with this kind of divergence. Since we are defining divergence by order in this paper, all of the previous results that we will cover do not, in their original paper, make use of this term. However, by inspecting their proofs it can be seen that the specific divergence they demonstrate is divergence by order. An important observation we will use in future proofs is that if $G$ has a subgraph with a sequence that diverges by order, then $G$ has a sequence that diverges by order.

We start by covering two conditions that are known to cause divergence by order. For the first one we need to define a specific class of graphs. Let $G^r_m$ be a unicyclic graph of order $m+r$ whose cycle has size $m$ and where one of the vertices in the cycle is adjacent to a pendent vertex of a path of order $r$. See Figure \ref{fig2.something} for an example. We have the following result due to Manjula in \cite{support1}.

\begin{figure}
    \centering
    \includegraphics[scale = 0.5]{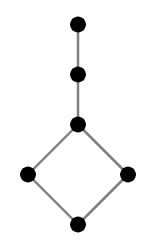}
    \caption{The graph $G^2_4$.}
    \label{fig2.something}
\end{figure}

\begin{theorem} \label{secondfamily}
\textbf{Manjula.} If $n = m+r$, then the sequence $\{HL^k(G^r_m) \}$ converges  to $C_{m+r}$ in $r$ iterations. Further, if $m + r>n$, then the sequence diverges by order.
\end{theorem}

Since this class of graphs $G^m_r$ where $m+r=n$ arises frequently, we will define the set that contains them. Let this family, which we denote by $\delta_n$, of graphs be defined as follows:
$$
\delta_n = \{G^r_m \text{ : }r+m=n \}.
$$
Notice that if $G\in \delta_n$, then any proper subgraph of $G$ terminates. Thus, $\delta_n \subset \lambda_n$.

The second known condition that implies divergence that we will use also requires us to define another class of graphs. For $m\geq 4$, define $F_m$ to be the graph of order $m$ and size $m+1$ consisting of a cycle of size $m$ chorded by an edge that joins two vertices whose distance is 2. The following result is due to  Chartrand \textit{et. al.} in \cite{main}.

\begin{theorem} \label{char1}
\textbf{Chartrand \textit{et. al. }} For $n\geq 4$ and $m\geq n$, the graph $F_m$ has a sequence that diverges by order.
\end{theorem}

We use these two previous results to prove a condition that implies divergence. Define the circumference of a graph $G$, denoted by $cr(G)$, as the size of the largest cycle in $G$, and where $cr(G) = 0$ if $G$ is a tree

\begin{theorem} \label{key1}
Let $G'\subseteq G$ be a connected subgraph. If $cr(G') = m \geq n$ but $G'\not \cong C_m$, then $G$ has a sequence that diverges by order.
\end{theorem}
\begin{proof}
Since it is enough for one single component of $G$ to have a sequence that diverges by order, we can assume $G' = G$. Note that since $G\not \cong C_m$, then there exists an edge $e=uv$ not in the cycle such that $v$ is in the cycle. There are two cases.

The first case is when $u$ is in the cycle. Let $G_0\subseteq G$ contain the cycle and the edge $e$. Further, set $N(u) = \{v, u_1, u_2 \}$ and $N(v) = \{u, v_1, v_2 \}$ such that the vertices are labeled as in Figure \ref{fig2.1}. Since $m\geq n$, the sequence $v_1, ..., u_1, u, v, v_2, ..., u_2$ is a path of order at least $n$. Thus, $e_1$ and $e_4$ are $P_n$ adjacent to $e$. By making a similar path, we notice that $e_2$ and $e_3$ are $P_n$-adjacent to $e$ as well. And so Figure \ref{fig2.2} gives a subgraph of $HL(G_0)$ which, as can be seen, is isomorphic to $F_{m+1}$. By Theorem \ref{char1}, a subgraph of $HL(G_0)$ has a sequence that diverges by order, and so $G$ has a sequence that diverges by order thus finishing this case.

\begin{figure}[!h]
\begin{center}
\includegraphics[scale=0.4]{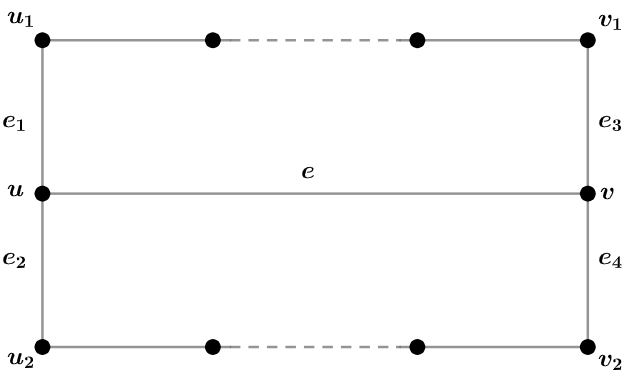}
\end{center}
\caption{The subgraph $G_0$ of $G$.}
\label{fig2.1}
\end{figure}

\begin{figure}[!h]
\begin{center}
\includegraphics[scale=0.45]{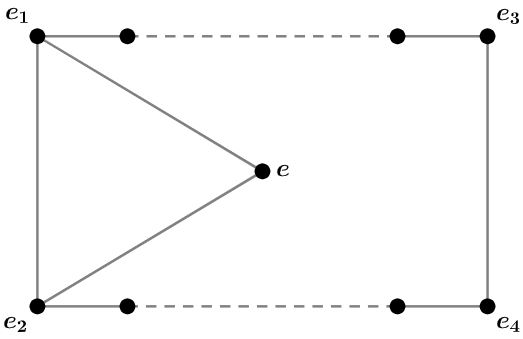}
\end{center}
\caption{A subgraph of $HL(G_0)$.}
\label{fig2.2}
\end{figure}

The second case is when $u$ is not in the cycle. This case, however, is a very simple case because then $G$ has $G^{1}_m$ as a subgraph. Since $m \geq n$, it follows that $m+1 > n$ so by Theorem \ref{secondfamily} the sequence of $G$ diverges by order.
\end{proof}

Theorem \ref{key1} is important because it shows that if the sequence of a connected graph $G\in \Lambda_n$ ever reaches a point where $HL^k(G)$ has a subgraph isomorphic to $C_m$ where $m \geq n$, then $HL^k(G)$ is in fact isomorphic to $C_m$. Although this condition is sufficient, we conjecture that it is also necessary.

\begin{conjecture}
The graph $G$ has a sequence that diverges by order if and only if there exists a $k$ such that $HL^k(G)$ has a connected subgraph $G'$ where $G'$ has a subgraph isomorphic to $C_m$ ($m\geq n$) but $G'\not \cong C_m$.
\end{conjecture}

Proving the above conjecture is just one step in understanding the structures that cause divergence. In particular, a characterization of the graphs $G$ whose sequence ends up satisfying the condition of Theorem \ref{key1} would be beneficial. 

We now provide another type of graph whose sequence diverges by order. Let the graph $CL(x, y, z)$ be the graph with order $x+y+z+1$ composed of three vertex disjoint paths of orders $x, y$ and $z$ respectively, where one pendent vertex of each path is adjacent to the same vertex. Observe that $CL(1, 1, 1)$ is the claw $K_{1, 3}$. 

\begin{theorem}\label{key2}
For every $n \leq 2k$ where $k + 1 < n$, the sequence of $CL(k, k, n-k-1)$ diverges by order.
\end{theorem}
\begin{proof}
Set $d=n-k-1$. We start by noticing that $HL(CL(k, k, d))$ is isomorphic to a unicyclic graph with $C_3$ as its cycle and where each vertex in the cycle is adjacent to a path of order $P_{k-1}$ or $P_{d-1}$. See Figure \ref{fig2.3} for $HL(CL(k, k, d))$ and its indexation.

\begin{figure}[!h]
\begin{center}
\includegraphics[scale=0.55]{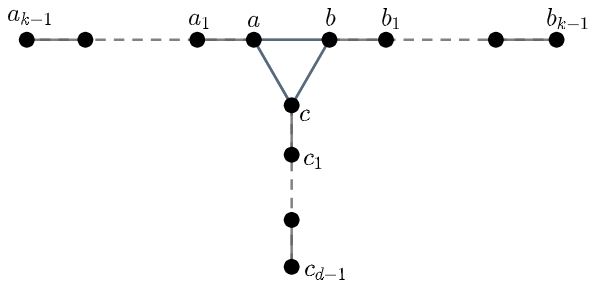}
\end{center}
\caption{The graph $HL(CL(k, k, n-k-1))$.}
\label{fig2.3}
\end{figure}

Notice that there is a path of order $(k-1) + 3 + (d-1)$ (which is equal to $n$) that includes the edges $a_1a$ and $ab$. Similarly, there exists a similar path of order at least $n$ that contains the edges $a_1a$ and $ac$. In general, we notice that $HL^2(CL(k, k, d))$ has the graph of Figure \ref{fig2.4} as a subgraph. 

\begin{figure}[!h]
\begin{center}
\includegraphics[scale=0.52]{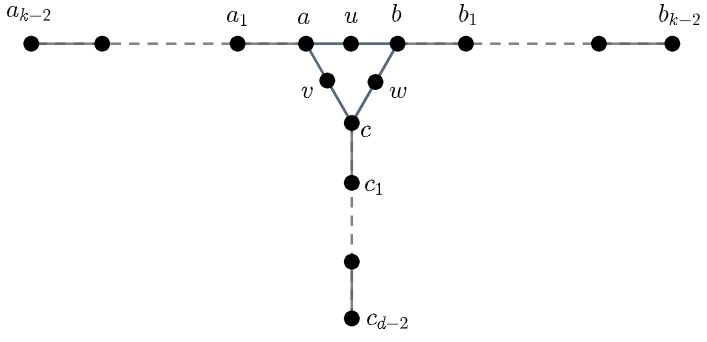}
\end{center}
\caption{A subgraph of $HL^2(CL(k, k, n-k-1))$.}
\label{fig2.4}
\end{figure}

It is important to remark that the graph of Figure \ref{fig2.4} is a subgraph of $HL^2(CL(k, k, d))$. In particular, the edges $uv$, $uw$ and $vw$ belong to $HL^2(CL(k, k, d))$. Nonetheless, this subgraph is enough to cause the sequence to diverge by order. Further, note that Figure \ref{fig2.4} does not use the same labelings that were used in Figure \ref{fig2.3}. For example, the vertex labeled as $a$ in Figure \ref{fig2.4} corresponds to the edge $a_1a$ in Figure \ref{fig2.3}. The case is similar for $b$ and $c$. We do the labeling this way so that we can illustrate better what will happen in $HL^m(CL(k, k, d)).$ Before going into these details, notice that in $HL^2(CL(k,k, d))$, there is a path of order $(k-2) + 5 + (d-2)$ (which is equal to $n$) that includes the edges $a_1a$ and $au$. Similarly, there is a path that contains $a_1a$ and $av$. We can make similar statements for vertex $b$.

In general, assume that the $m^{\text{th}}$ iteration of the sequence has a unicyclic subgraph with three vertices $a, b,$ and $c$ in the cycle, where each of $a$ and $b$ are adjacent to a path of order $k-m$, and $c$ is adjacent to a path of order $d-m$. Further, assume that any two vertices in $\{a, b, c \}$ have distance $m$. If $a_1$ is the vertex adjacent to $a$ in the path, then $a_1a$ will be in a path of order $(k-m) + (m + 1 + m) + (d-m)$, which is equal to $n$. Through similar arguments for the other vertices adjacent with $a$, and by repeating this with $b$ and $c$, we conclude that the $(m+1)^{\text{th}}$ iteration will have a subgraph with these same properties. Finally, since $HL(CL(k, k, d))$ has this property, then every graph in the sequence up to the $d^{\text{th}}$ iteration has it. 

This is enough to prove divergence by order because then $HL^d(CL(k, k, d))$ will have a subgraph isomorphic to $G^{k-d}_{3d}$. Since $k-d+3d > n$, Theorem \ref{secondfamily} guarantees that the sequence of this subgraph diverges by order and thus the sequence of $CL(k, k, n-k-1)$ diverges by order too.
\end{proof}

\begin{corollary}\label{badclaws}
Let $v$ be the vertex with degree $3$ in $CL(x, y, z)$, where $x, y, $ and $z$ are integers. If the edges incident to $v$ are pairwise $P_n$-adjacent, then $CL(x, y, z)$ has a sequence that diverges by order. 
\end{corollary}
\begin{proof}
Let $P_1, P_2$ and $P_3$ be the three paths joined by the vertex $v$ where $|V(P_1)| = x, |V(P_2)|=y$, and $|V(P_3)|=z$. Notice that $x + y + 1 \geq n$, so $y \geq n - 1 - x$. For now, assume that $y = n - x - 1$. Since $y + z + 1 \geq n$, it follows that $z \geq x$. As a consequence, the fact that $z + x + 1 \geq n$ implies that $2x \geq n$. Thus, $CL(x, y, z)$ has as a subgraph $CL(x, n-x-1, x)$ where $2x \geq n$. By Theorem \ref{key2}, the claw has a sequence that diverges by order. If it is the case that $y > n - x - 1$, then $CL(x, n- x - 1, z)$ is a subgraph and the same proof applies.
\end{proof}

We will prove one more condition that implies divergence by order. For it, we need one more result due to Chartrand \textit{et. al.} in \cite{main}.

\begin{theorem} \label{char2}
\textbf{Chartrand et. al.} If $G$ is a connected graph, then $HL(G)$ contains at most one component that is not an isolated vertex.
\end{theorem}

We remind the reader that two graphs $G_1$ and $G_2$ are not equal if and only if $V(G_1) \not = V(G_2)$ or $E(G_1) \not = E(G_2)$, and that this is possible even if $G_1 \cong G_2$. Our next result is that the sequence for $G$ will diverge by order if $G$ contains two distinct subgraphs from the family $\delta_n = \{G^r_m \text{ : }r+m=n \}$.

\begin{theorem}
Let $G_1, G_2\in \delta_n$ be subgraphs of the same component of $G$. If $G_1 \not = G_2$, then $G$ diverges by order.
\end{theorem}
\begin{proof}
We may assume that $G$ is connected. Set $G_1 \cong G^{r_1}_{m_1}$ and $G_2 \cong G^{r_2}_{m_1}$. We first consider the case where $G_1\not \cong G_2$. Without loss of generality assume that $r_1 < r_2$. Theorem \ref{secondfamily} implies that $HL^{r_1}(G_1) \cong C_n$. Since $r_1< r_2$, we have that $HL^{r_1}(G_2) \not \cong C_n$. Thus $HL^{r_1}(G)$ will have a subgraph isomorphic to $C_n$, but since the sequence of $G_2$ does not terminate, it follows that $HL^{r_1}(G) \not \cong C_n$. By Theorem \ref{key1}, the sequence diverges by order.

Assume then that $G_1 \cong G_2$. Set $m = m_1 = m_2$ and $r = r_1 = r_2$. Note that $HL^r(G_1) \cong HL^r(G_2) \cong C_n$, and so every edge in both $G_1$ and $G_2$ will be the vertices of a cycle of size $n$ in $HL(G)$. Further, observe that if $E(G_1) = E(G_2)$, then $V(G_1) = V(G_2)$ since $G_1\cong G_2$, so it must be that $E(G_1)\not = E(G_2)$. Thus, $HL^r(G)$ will contain two different cycles of size $n$ in the same component (we know that they are in the same component by Theorem \ref{char2}). This satisfies the condition of Theorem \ref{key1}, and so $G$ has a sequence that diverges by order.
\end{proof}

The natural generalization of this theorem is: if $G$ has a component containing distinct subgraphs $G_1$ and $G_2$ such that $G_1, G_2\in \Lambda_n$, then $G$ has a sequence that diverges by order. Nonetheless, Figure \ref{fig2.5} shows a counterexample to this. For a conjecture, we need to make one of these conditions stronger.

\begin{conjecture}
Let $G$ be a connected graph, and let $G_1$ and $G_2$ be subgraphs of $G$. If $G_1 \not \cong G_2$ and $G_1, G_2\in \Lambda_n$,  then $G$ has a sequence that diverges by order.
\end{conjecture}

\begin{figure}[!h]
\begin{center}
\includegraphics[scale=0.5]{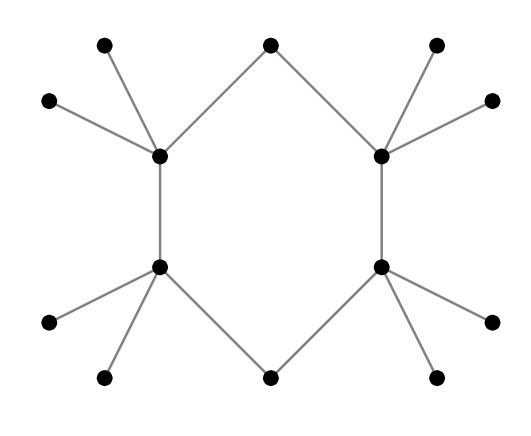}
\end{center}
\caption{A graph whose sequence converges when $H\cong P_8$.}
\label{fig2.5}
\end{figure}

\section{Properties of minimally $n$-convergence}

\subsection{Basic properties}

We start by reminding the reader of the definition of minimally $n$-convergence.

\begin{definition}
A graph $G$ is said to be minimally $n$-convergent if $G\in \Lambda_n$ but every proper subgraph of $G$ is not in $\Lambda_n$. Further, let $\lambda_n$ be the set of all minimally $n$-convergent graphs.
\end{definition}

Notice that every graph $G\in \Lambda_n$ has a subgraph $G'\subseteq G$ such that $G'\in \lambda_n$. This is why studying minimally $n$-convergent can be far reaching: obtaining properties about graphs in $\lambda_n$ gives us properties about some subgraph of every graph in $\Lambda_n$. We will spend the rest of the paper developing results related to minimally $n$-convergence. Although the next result is easy to obtain, it provides a template for how to use the definition of  minimally $n$-convergence in proofs.

\begin{lemma} \label{key3}
If $G\in \lambda_n$, then every edge in $G$ is in a copy of $P_n$.
\end{lemma} 
\begin{proof}
For a contradiction, assume that the edge $e$ is not in a $P_n$ (i.e. it is not in a subgraph isomorphic to $P_n$). First, notice that $e$ is an isolated vertex in $HL(G)$, as otherwise there would exists an edge $f\in E(G)$ such that $e$ is $P_n$-adjacent to $f$, meaning that there exists a path $P_n$ containing both $e$ and $f$ which would contradict our assumption that $e$ is not in a $P_n$. Thus, the edge $e$ is an isolated vertex in $HL(G)$. From $G\in \Lambda_n$, we know that $HL(G)\in \Lambda_n$, and since $e$ is an isolated vertex in $HL(G)$, it follows that $HL(G)-e\in \Lambda_n$ also. 

Since $e$ is not $P_n$-adjacent to $f$ for every $f\in E(G)$, it follows that $E(HL(G-e)) =E(HL(G) - e)$. Also, from the definition of $H$-line graphs we also know that $V(HL(G-e)) = V(HL(G)-e)$, so we conclude that $HL(G-e) = HL(G)-e$. But then $HL(G-e)\in \Lambda_n$, so $G-e\in \Lambda_n$. This contradicts that $G$ is minimally $n$-convergent thus finishing the proof.
\end{proof}

Lemma \ref{key3} confirms the notion that every edge in $G$ is needed.

\begin{lemma}\label{key6}
Let $G\in \lambda_n$, and assume both that $G\not \in \delta_n$ and that $G$ is not a cycle. Further, let $P$ be a non-extendable path in $G$ that has order at least $n$. If $p_1$ and $p_2$ are pendent vertices in $P$, then $p_1$ and $p_2$ are pendent vertices in $G$.
\end{lemma}
\begin{proof}
It is enough to consider the case where $G$ is connected. If $G$ is disconnected, the proof applies to one of its components. Assume, for a contradiction, that $p_1$ is not a pendant vertex in $G$. Since $P$ cannot be extended, then $p_1$ must be adjacent to some vertex $p$ in $P$. If $p = p_2$, then $G$ has a cycle of size at least $n$. By hypothesis, $G$ is not a cycle so $G\not \cong C_n$, so by Theorem \ref{key1}, it follows that $G$ has a sequence that diverges by order, which is a contradiction. If $p \not = p_2$, then the subgraph of $P$ with the edge $p_1p$ is a subgraph $G_0$ isomorphic to $G^r_m$ for some $r$ and $m$. Since $P$ has order at least $n$, it follows that $r+m \geq n$. However, it cannot be the case that $r+m>n$ since $G$ would have a sequence that diverges in order by Theorem \ref{secondfamily}. So $r+m = n$. However, Theorem \ref{secondfamily} implies that $G_0\in \Lambda_n$. And since $G\not \in \delta_n$, then $G_0$ is a proper subgraph of $G$. This contradicts that $G$ is minimally $n$-convergent thus finishing the proof.
\end{proof}

\begin{lemma} \label{newsupport}
If $G\in \lambda_n$, then $HL(G)$ has the same number of components as $G$.
\end{lemma}
\begin{proof}
Every component in $G$ has every edge in a $P_n$ by Lemma \ref{key3}, so if $G'$ is a component of $G$, it is not possible for $HL(G')$ to have an isolated vertex. Thus, by Theorem \ref{char2}, $HL(G')$ is connected. In other words, every component of $G$ generates exactly one component in $HL(G)$.
\end{proof}

\subsection{Minimally $n$-convergence and unicyclicness}

Knowing the structure of minimally $n$-convergent graphs can facilitate multiple proofs. We have the following conjecture about the structure of these graphs. 

\begin{conjecture} \label{unicyclicMin}
If $G\in \lambda_n$, then $G$ has unicyclic components.
\end{conjecture}

Establishing this conjecture can be very helpful when proving statements about minimally $n$-convergent graphs as it provides us with one and only one cycle to work with. Further, since every graph $G\in \Lambda_n$ has a subgraph in $\lambda_n$, Conjecture \ref{unicyclicMin} would prove that no tree has a convergent sequence. We will use the rest of the paper to give results related to this conjecture. For this purpose, we need to study more carefully the relationship between $H$-line graphs and the property of unicyclicness. 

\begin{definition}
Let $C$ be the unique cycle in a unicyclic graph $G$.
\begin{itemize}
    \item The subgraph $A$ of $G$ is called an arm if $A$ is a component of $G-C$. 
    \item The armset of $G$, denoted by $\mathcal{A}(G)$, is the set
$$
\mathcal{A}(G):=\{A \text{ : }A \text{ is an arm of G} \}.
$$
    \item The vertex $r\in V(C)$ is a called a root if $r$ is adjacent to some vertex in an arm $A\in \mathcal{A}(G)$.
    \item The root identifier function, denoted by $\mathcal{A}_G: \mathcal{A}(G) \rightarrow V(C)$, is the function that takes $A$ to the unique root that is adjacent to some vertex in $A$.
\end{itemize}
\end{definition}

Note that $\mathcal{A}_G$ is well defined because if there were two roots $r_1$ and $r_2$ associated with an arm $A$, then $G$ would not be unicyclic to begin with. Further, it is not necessary for $\mathcal{A}_G$ to be a one-to-one function. In particular, there exists unicyclic graphs $G\in \Lambda_n$ that have roots adjacent to multiple roots. For instance, the graph in Figure \ref{fig2.5} has 4 roots but 8 arms. We need a result due to Britto-Pacumio in \cite{limits}. 

\begin{theorem} \label{limits2}
\textbf{Britto-Pacumio. }If $G$ is unicyclic and every edge of $G$ is in a $P_n$, then $cr(HL(G))\geq cr(G)$.
\end{theorem}

\begin{corollary}
Let $G\in \lambda_n$. If $G$ is unicyclic, then $HL(G)$ is not a tree. 
\end{corollary}

The proof of the above corollary is immediate from Lemma \ref{key3} and Theorem \ref{limits2}. The rigid structures of unicyclic graphs allows for many proof techniques that make use of roots and arms. 

\begin{lemma} \label{finalkey}
Let $G\in\Lambda_n$ such that $G$ is unicyclic. If $e$ is an edge in an arm of $G$, then $e$ cannot be in a cycle of $HL(G)$. 
\end{lemma}
\begin{proof}
Let $C$ be the unique cycle in $G$. For a contradiction, assume that there exists an arm $A\in \mathcal{A}(G)$ and a cycle $C'$ in $HL(G)$ such that $e\in E(A)$ and $e\in V(C')$. Set $C': e_1, ... , e_p, e_1$ where $e_i\in E(G)$. Without loss of generality, assume that $e = e_1$, and let $v$ be the vertex in $G$ incident to both $e_1$ and $e_2$. Since both vertices incident to $e$ are in $A$, we have that $v\in V(A)$. Let $f$ be the edge incident to $\mathcal{A}_G(A)$, the root of $A$, and to some vertex $u$ in $A$, the vertex in $A$ adjacent to the root. 

Notice that $e_i \not = e_j$ for $i \not = j$. If there exists an $i$ such that $e_i\in V(C)$, then $f$ would be in the sequence $e_1, e_2,..., e_i$. However, $f$ would also need to be in the sequence $e_i, e_{i+1}, ..., e_{p}, e_1$. Since $f$ is not in the arm, we have that $f \not = e_1$ and $f\not = e_i$, so then we have a repeated element in the sequence $e_1,..., e_p$, which is a contradiction. Thus, $e_i\not \in V(C)$ for every $i$. In other words, every vertex of the cycle $C'$ must be either in $A$ or be $f$. Since $A$ is a tree, we have that every edge in $V(C')$ must be incident to the same vertex as otherwise we can craft a similar argument to the case where there is an edge in the cycle of $G$. So every edge $e_i$ is incident to $v$. 

\begin{figure}[!h]
    \begin{center}
    \includegraphics[scale=0.65]{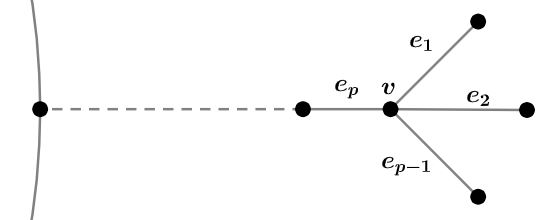}
    \end{center}
    \caption{An illustration of a subgraph of $G$.}
    \label{fig3.2}
\end{figure}

If $p=3$, then the graph contains a claw where the edges incident to $v$, which would be  $\{e_1, e_2, e_3 \}$, are pairwise $P_n$-adjacent. By Corollary \ref{badclaws}, the graph has a divergent sequence, which is a contradiction. Thus, $p > 3$. There exists a unique path between $v$ and $\mathcal{A}_G(A)$. Without loss of generality, assume that this path contains the edge $e_p$. See Figure \ref{fig3.2}. Let $e_i = vv_i$ for every $i$. Let $P_i$ denote the longest path that has as an endpoint $v_i$ and which does not contain the edge $e_i$. For $e_1, e_2$ and $e_{p-1}$, denote the order of those paths by $k, k'$ and $k''$ respectively. Since $e_1$ is $P_n$-adjacent to $e_p$, the order of $P_p$ must be at least $n-k-1$. Since $e_p$ is $P_n$-adjacent to $e_{p-1}$, it must be that $n-k-1 + k'' \geq n$, or that $k'' \geq k+1$. It cannot be the case that $e_1$ is $P_n$-adjacent to $e_{p-1}$ because then the set $\{e_1, e_p, e_{p-1} \}$ would be a set of pairwise $P_n$-adjacent edges in a claw, and Corollary \ref{badclaws} would give a contradiction. Thus, $n > k + k'' + 1\geq k + k+1+1$, or that 
\begin{equation}\label{handy}
    n > 2k + 2.
\end{equation}
Since $e_1$ is $P_n$-adjacent to $e_2$, it follows that $k + k' + 1 \geq n$, or that $k' \geq n-k-1$. Before we noted that $k'' \geq k+1$, so $k' + k'' +1 \geq n-k-1 + k+1+1 = n+1$. Thus, $e_2$ is $P_n$-adjacent to $e_{p-1}$ Finally, notice that 
\begin{eqnarray*}
1+k' + n-k-1 & \geq & 1 + n-k-1 + n-k-1 \\
            & = & 2n -2k -1 \\
            & > & n
\end{eqnarray*}
where the last inequality is obtained from equation (\ref{handy}). Thus, $e_2$ is $P_n$-adjacent to $e_p$, and so $\{e_2, e_p, e_{p-1} \}$ is a set of pairwise $P_n$-adjacent edges in a claw. By Corollary \ref{badclaws}, $G$ has a sequence that diverges by order, which is the contradiction, which finishes the proof.
\end{proof}
\begin{corollary}\label{smallkey}
Let $G\in \Lambda_n$ such that $G$ is unicyclic. If $r$ is a root, then the edges incident to $r$ cannot induce, as vertices, a graph with a cycle of order $4$ or more in $HL(G)$.
\end{corollary}

The corollary follows from the proof technique used with the case where $p>3$ and the fact that a claw that satisfies the condition of Corollary \ref{badclaws} can still be obtained. We now have all the tools needed for our last result. Remember that the girth of a graph $G$, denoted by $g(G)$, is the size of the smallest cycle in $G$.

\begin{theorem}
Let $G\in \lambda_n$ such that $g(HL(G))> 4$. If $G$ has unicyclic components, then $HL(G)$ has unicyclic components.
\end{theorem}
\begin{proof}
We may, again, assume $G$ is connected as the proof applies to each components of $G$ if $G$ is disconnected. Since $G$ is unicyclic and is in $\lambda_n$, it follows that $HL(G)$ must have a cycle. 

For a contradiction, assume that there exists two cycles $C_1$ and $C_2$ in $HL(G)$ such that $C_1\not = C_2$. Corollary \ref{smallkey} and $g(HL(G))>4$ imply that no root is incident to every edge of $C_1$ or $C_2$. And since no edge in the arms can be in a cycle of $HL(G)$ it must be that every edge in the cycle of $G$ is in the cycles of $HL(G)$. Set $C$ as the cycle of $G$, so $E(C)\subseteq V(C_1)$ and $E(C) \subseteq V(C_2)$. Since $C_1\not = C_2$, there must exists an edge $e\in C_1$ that is not in $C_2$. This edge cannot be in $C$ so it is incident to a root $r$. Set $E(r)$ as the set of edges incident to $r$, and let $E(r) \cap E(C) = \{f_1, f_2 \}$. But $f_1$ and $f_2$ are both in $V(C_1)$ and $V(C_2)$, so 
$$
(E(r)\cap V(C_1)) \cup (E(r)\cap V(C_2))
$$
induces at least one cycle in $HL(G)$. Every edge in this cycle is incident to $r$, which contradicts Corollary \ref{smallkey}.
\end{proof}

Assuming that $g(HL(G))>4$ is most likely not necessary for the statement to remain true. Finding a proof that avoids using this assumption is desirable, but probably hard. To continue the study of minimally $n$-convergent graphs, we propose two directions. The first one is working more towards the proof of Conjecture \ref{unicyclicMin}. The second one, which has not been discussed in detail in this paper, is establishing the veracity of the following conjecture. 

\begin{conjecture} \label{bridge}
If $G\in \Lambda_n$ and $G$ is not the disconnected union of two graphs in $\Lambda_n$, then there exists a unique graph $G'$ in $\lambda_n$ such that $G'\subseteq G$.
\end{conjecture}

The existence of $G'$ is already known. The conjecture adds that this graph is unique. This would imply that graphs with convergent sequences are actually just variations of graphs in $\lambda_n$. In other words, characterizing $\Lambda_n$ would heavily depend on characterizing $\lambda_n$. Proving Conjecture \ref{bridge}, however, needs a more thorough development of the theory of minimally $n$-convergence.

\section{Acknowledgements}

I want to thank first and foremost Dr. Michelle Robinette, my undergraduate thesis advisor. Her support in this project and other endeavours were key to my professional development. I also want to thank the other members of my thesis committee Dr. Bryan Bornholdt and Dr. Douglas Burke for their time and helpful comments. This research did not receive any specific grant from funding agencies in the public, commercial, or not-for-profit sectors.

\end{document}